\documentclass[10pt,leqno]{amsart}
\usepackage{graphicx}
\usepackage{indentfirst,csquotes}

\usepackage{amssymb,amsthm,amsmath}
\usepackage{xcolor,paralist,hyperref,titlesec,fancyhdr,etoolbox}
\newtheorem{theorem}{Theorem}[]

\newtheorem{example}[theorem]{Example}
\newtheorem{lemma}[theorem]{Lemma}
\newtheorem{app}[theorem]{Application}

\newtheorem{corollary}[theorem]{Corollary}

\titleformat{\section}[display]{\normalfont\huge\bfseries\centering}{\centering\chaptertitlename\thechapter}{10pt}{\Large}
\titlespacing*{\section}{0pt}{0ex}{0ex}

\hypersetup{ colorlinks=true, linkcolor=black, filecolor=black, urlcolor=black }

\usepackage{lipsum}

\begin{document}
\title{Effects of Endomorphisms on the commutativity of Banach algebras} 
\author[]{Mohamed MOUMEN and Lahcen  TAOUFIQ}
\date{\today}
\address{National School of Applied Sciences
		\\Ibn Zohr University
		\\ Agadir\\ Morocco}
\email{mohamed.moumen@edu.uiz.ac.ma, l.taoufiq@uiz.ac.ma}
\maketitle

\let\thefootnote\relax
\footnotetext{MSC2020: 46J05, 16S50, 16N60.} 

\begin{abstract}
	In this article, we are interested in endomorphism's impact on the commutativity of a Banach algebra. Our research uses a topological approach, drawing on specific results from functional analysis and algebraic techniques.\;Also, we provide examples to show that we need certain assumptions.
\end{abstract} 

\bigskip
	
	{\bf 1. Introduction}
 
	In this paper, $\mathcal{A}$ denotes a Banach algebra with center  $Z(\mathcal{A})$.\;For all $x, y \in \mathcal{A},$ the symbol $[x, y]$ (resp.\;$x\circ y$)  will indicate the Lie product   $xy - yx$ (resp.\;Jordan product $xy+yx$). Recall that, $\mathcal{A}$ is prime if for any $(x,y)\in \mathcal A^2,$  $x\mathcal{A}y=0$ then either $x=0$ or $y=0$.
	An additive maps $d:\mathcal{A}\to \mathcal{A}$ is said to be  derivation if $d(xy) = d(x)y + xd(y)$ for all $x, y \in \mathcal{A}.$
	
	\hspace{0.5cm} 
	
	Numerous researchers have developed commutativity requirements for specific rings by considering circumstances that appear too weak to include commutativity. In this regard, Yood  \cite{Yood94}  has demonstrated that a prime complex Banach algebra $\mathcal{A}$ should be commutative if there are non-void open subsets $\mathcal H_1$ and  $\mathcal H_2$
	of $\mathcal{A}$ such that for any $(x,y) \in \mathcal H_1\times \mathcal H_2$  there are positive integers
	$p = p(x, y),\;q = q(x, y)$ depending on $x$ and $y,\;p > 1, q > 1,$ such
	that either $[x^p,y^q]\in Z(\mathcal{A})$ or $x^py^q \in Z(\mathcal{A}).$  Inspired by Yood's result,  Mohamed Moumen, Lahcen Taoufiq and Lahcen Oukhtite \cite{rkh} proved that if $\mathcal{A}$ is a prime Banach algebra admitting a  continuous  derivation $d$ such that  $d(x^{n}y^{m})+ [x^{n},y^{m}]\in Z(\mathcal{A})$ where  $ n = n(x, y)$ and $ m = m(x, y)$ are
	integers for sufficiently many $x$ and $y$ then $\mathcal{A}$ is commutative
	( see \cite{rkh2} and \cite{r1} for further information and examples).\\
	
	\hspace{0.5cm} 
	
	Considering these results, this paper aims to study the commutativeness of a Banach algebra $\mathcal{A}$ equipped with a continuous surjective endomorphism satisfying some specific algebraic identities on open subsets of $\mathcal{A}.$ Our topology strategy relies on Baire's category theorem and some functional analysis properties. Among our results,
	we have shown that a prime Banach algebra $\mathcal{A}$ must be commutative if for any non-empty open subsets $\mathcal{H}_{1},\mathcal{H}_{2}$ and a continuous surjective endomorphism  $f$ and any $(x,y)\in \mathcal{H}_{1}\times \mathcal{H}_{2}$ there are strictly positive integers $n, m$ such that $f(x^{n}\circ y^{m})+ x^{n}y^{m}\in Z(\mathcal{A}).$ \\
	
	{\bf 2. Main results}
	\hspace{0.5cm} 
	Throughout this article, $\mathbb{R}$ will be  the field of real
	numbers, and $\mathbb C$ will be the field of complex numbers.\\

	We start our discussion with the following well-known lemma, due to  Bonsall and  Duncan {\rm \cite{Bonsal}}, who is essential for creating the evidence for our fundamental theorems.
	\begin{lemma}\label{l1}
		
		Let $ \mathcal{A} $ be a real or complex Banach algebra and $ T(t)= \displaystyle\sum_{i=0}^{n} t^{i} a_{i} $ a polynomial in the real
		variable $t$ with coefficients in $\mathcal A.$  If  for an infinite set of real values of $t,\;T (t) \in C,$
		where $C$ is a closed linear subspace of $\mathcal A,$ then every $a_i$ lies in $C.$
		
	\end{lemma}
	\hspace{0.5cm} 
	
	\noindent The following theorem is the first main result in this article.
	\begin{theorem}\label{a}
		If a continuous surjective endomorphism $f$ of a unitary prime Banach algebra
		$\mathcal{A}$ satisfying:
		for each $(x,y)\in \mathcal{H}_{1}\times \mathcal{H}_{2}$ there are strictly positive integers $p,q$ such that
		$\;f(x^{p}y^{q})+[x^{p},y^{q}]\in Z(\mathcal{A}),$
		then $ \mathcal{A}$ must be  commutative, where  $\mathcal{H}_{i}$  is a non-void open subset of $\mathcal{A}$ ( $i=1$ or $i=2$).
	\end{theorem}
	\begin{proof}
		For all $(p,q)\in\mathbb N^{*}\times \mathbb N^{*}$
		let's define the next ensembles:
		$$ O_{p,q}=\{(x,y)\in \mathcal{A}^{2} \; \mid \;  f(x^{p}y^{q})+ [x^{p},y^{q}]\notin Z(\mathcal{A} )\}$$    $$\mbox{and}$$     $$ F_{p,q}=\{(x,y)\in \mathcal{A}^{2} \; \mid \;  f(x^{p}y^{q})+ [x^{p},y^{q}]\in Z(\mathcal{A} )\}.$$
		Evidently $(\cap O_{p,q})\cap (\mathcal H_{1} \times \mathcal H_{2})= \varnothing,$ furthermore there exists $(x,y)\in\mathcal H_{1} \times\mathcal H_{2} $ such that  $f(x^{p}y^{q})+ [x^{p},y^{q}]\notin Z(\mathcal{A})$
		for any $(p,q)\in\mathbb N^{*}\times \mathbb N^{*},$  a contradiction.\\
		Now we prove that each $O_{p,q}$ is open in $\mathcal{A}\times\mathcal{A}$  or equivalently its complement $F_{p,q}$ is closed. For this, we consider a sequence $((x_{k},y_{k}))_{k\in\mathbb{N}}\subset F_{p,q}$ converging to $(x,y)\in\mathcal{A}\times\mathcal{A}$.
		Since $((x_{k},y_{k}))_{k\in\mathbb{N}}\subset F_{p,q}$, therefore
		\begin{center}
			$f((x_{k})^{p}(y_{k})^{q})+ [(x_{k})^{p},(y_{k})^{q}]\in Z(\mathcal{A})$ for all $k\in\mathbb{N}.$
		\end{center}
		By the assumption $f$ is continuous, limit of the sequence $(f((x_{k})^{p}(y_{k})^{q})+ [(x_{k})^{p},(y_{k})^{q}])_{k\in\mathbb{N}}$ it exactly $ f(x^{p}y^{q})+ [x^{p},y^{q}] $ and the fact that  $Z(\mathcal{A})$ is closed  ensures $ f(x^{p}y^{q})+ [x^{p},y^{q}]\in Z(\mathcal{A}).$  Therefore, $ (x,y)\in F_{p,q} $ means $ F_{p,q}$ is closed that is $O_{p,q}$ is open.\\
		In light of Baire category theorem, if every $ O_{p,q} $ is
		dense then their intersection is also dense, which  contradicts the existence of  $\mathcal H_{1}$ and $\mathcal H_{2}.$ Accordingly, there are positive integers $n, m$  such that  $ O_{n,m} $ is not dense in $\mathcal{A}$ which forces existence of a nonvoid open subset $ O\times O^{'}$ in $ F_{n,m}$
		such that
		\begin{center}
			$f(x^{n}y^{m})+ [x^{n},y^{m}]\in Z (\mathcal{A})$ for all $x\in O$, for all $y\in O^{'}.$
		\end{center}
		Fix $ y\in \mathcal O^{'}$. Let $x \in \mathcal O $ and $z\in\mathcal{A},$ then $x+tz\in \mathcal O$ for all sufficiently small real $t.$ Therefore
		$$
		P(t)= f((x+tz)^{n}y^{m})+[(x+tz)^{n},y^{m}] \in Z(\mathcal{A}).
		$$
		Since $P(t)$ may be written as
		$$
		P(t)=A_{n,0}(x,z,y)+tA_{n-1,1}(x,z,y) +t^{2} A_{n-2,2}(x,z,y)+...+t^{n}A_{0,n}(x,z,y)
		$$
		where the last term in this polynomial is $A_{0,n}(x,z,y)=f(z^{n}y^{m})+[z^{n},y^{m}],$  then Lemma \ref{l1}
		yields
		$$
		A_{0,n}(x,z,y)=f(z^{n}y^{m})+[z^{n},y^{m}]\in Z(\mathcal{A}).
		$$
		Consequently, given $ y\in \mathcal O^{'}$ we get
		$$
		f(x^{n}y^{m})+[x^{n},y^{m}]\in Z(\mathcal{A})\;\;\mbox{for all}\;\;x\in \mathcal{A}.
		$$
		Using a similar approach, we arrive at
		$$
		f(x^{n}y^{m})+[x^{n},y^{m}]\in Z(\mathcal{A})\;\;\mbox{for all}\;\;(x,y)\in \mathcal A^{2}.
		$$
		In particular,  $ f(x^{n+m})\in Z(\mathcal{A})$ for all $ x\in \mathcal{A}$. We can write
		$$
		f((e+tx)^{n+m})=\displaystyle\sum_{k=0}^{n+m} \binom {n+m}{k}t^{k}f( e^{n+m-k}x^{k})\in Z(\mathcal{A}).
		$$
		By Lemma  \ref{l1}, we have $ \binom {n+m}{k}f(e^{n+m-k}x^{k})\in Z(\mathcal{A})$ for all $0\leq k \leq n+m$ proving that $f(e^{n+m-1}x)\in Z(\mathcal{A}).$ Consequently, $f(x)\in Z(\mathcal{A})$,   that is $f(\mathcal{A})\subset Z(\mathcal{A}).$
		Since $f$ is surjective, then the assertion of the theorem.
	\end{proof}
	\noindent
	
	\hspace{0.5cm} 
	
	As two applications of Theorem \ref{a}.
	\begin{app}
		Let $n\geq 2$ be an integer and $\mathcal{H}$
		the set of all nil-potent matrices in $\mathcal M_n(\mathbb R)$ (the space of all $n$ by $n$ matrices with entries from $\mathbb{R}$).\;$\mathcal M_n(\mathbb R)$ endowed with the usual operations on matrices and the norm $\|.\|_1$ defined by 
		$\|A\|_1 =\displaystyle\sum_{1\leq i,j\leq n}|a_{ij}|$ for all $A=(a_{i,j})_{1\leq i,j\leq n},$ is  a  non-commutative prime Banach Algebra.\;Observe that,\;for every continuous automorphism $f$ of $\mathcal M_n(\mathbb R)$,
		for all $(x,y)\in \mathcal H^2$ there exists $(n,m)\in \mathbb N^2$ and $ f(x^{n}y^{m})+[x^{n},y^{m}]=0$.
		Hence,\;it follows from Theorem \ref{a} that $\mathcal{H}$ is not open in $\mathcal M_n(\mathbb R)$.
	\end{app}
	\hspace{0.5cm} 
	
	\begin{app}
		Consider the real prime Banach algebra $\mathcal{A}=\mathcal{\mathcal{M}}_{2}(\mathbb{C})$
		endowed with the norm $\Vert .\Vert_{1}$ defined by
		$$ \left\Vert \begin{pmatrix}
			a & b\\
			c & d
		\end{pmatrix}
		\right\Vert_{1}= |a|+|b| +|c|+|d|.
		$$
		Let $f: \mathcal A \longrightarrow \mathcal{A}$ be the mapping defined by $$f(\begin{pmatrix}
			a & b\\
			c & d
		\end{pmatrix})=\begin{pmatrix}
			3a-b+6c-2d & -2a+b-4c+2d\\
			3a-b+9c-3d & -2a+b-6c+3d
		\end{pmatrix}$$ for every $(a,b,c,d) \in\mathbb C^4,$ then $f$ is a continuous automorphism of $\mathcal{A}.$\\
		Now, we put 
		$ \mathcal{F}=\left\{
		\begin{pmatrix}
			t & 0 \\
			0 & t
		\end{pmatrix} \;|\; t^2=1 \; \mbox{and}\; t\in \mathbb{C}  \right\}$.\;Moreover,
		for $ A=\begin{pmatrix} a & 0 \\
			0 & a\end{pmatrix} \in\mathcal{F}$, $ B=\begin{pmatrix} b & 0 \\
			0 & b\end{pmatrix} \in\mathcal{F}$  and  for all  $ (m,n)\in\mathbb{N}^{2}$ we  have
		$$
		A^{n}B^{m}=
		\begin{pmatrix}
			a^{n}b^{m} & 0 \\
			0 & a^{n}b^{m}
		\end{pmatrix}
		,\;\;
		[A^{n},B^{m}]=
		\begin{pmatrix}
			0 & 0 \\
			0 & 0
		\end{pmatrix}$$
		consequence of which
		$$
		f(A^{n}B^{m})+[A^{n},B^{m}]\in Z(\mathcal{A}).
		$$
		It follows that the automorphism $f$ satisfies the hypotheses of  Theorem \ref{a}, however, $\mathcal{A}$ is not commutative. Consequently, $\mathcal{F}$ is not open in $\mathcal{A}$. 
	\end{app}
	\hspace{0.5cm} 
	
	In the following theorem, the endomorphism $f$ is not necessarily surjective.
	\begin{theorem}\label{b}
		Consider that in a unitary prime Banach algebra $\mathcal{A}$,  $\mathcal{H}_{1}$ and $\mathcal{H}_{2}$ two non-empty open subsets. If $\mathcal{A}$ has a continuous endomorphism $f$ satisfying:
		for all $(x,y)\in \mathcal{H}_{1}\times \mathcal{H}_{2}$ there is  $(p,q)\in\mathbb N^{*}\times \mathbb N^{*}$ and
		$f(x^{p}\circ y^{q})+x^{p}y^{q}\in Z(\mathcal{A}),$
		then A must be commutative. 
	\end{theorem}
	\begin{proof}
		For $(p,q)\in\mathbb N^{*}\times \mathbb N^{*}$ let us consider
		$$ O_{p,q}=\{(x,y)\in \mathcal{A}^{2}\;|\; f(x^{p}\circ y^{q})+ x^{p}y^{q}\notin Z(\mathcal{A} )\}\;\mbox{and}\;F_{p,q}=\{(x,y)\in \mathcal{A}^{2}\;|\; f(x^{p}\circ y^{q})+ x^{p}y^{q}\in Z(\mathcal{A} )\}.$$
		Using methods similar to those used in the demonstration of Theorem \ref{a}, we can prove the existence of positive integers $n$ and $m$ such that
		\begin{equation}\label{equ1}
			f(x^{n}\circ y^{m})+ x^{n}y^{m}\in Z(\mathcal{A})\;\;\;\mbox{for all}\;\; x,y\in\mathcal{A}.
		\end{equation}
		Replacing $x$ by $x^{m}$ and $y$ by $y^{n}$ in $(\ref{equ1}),$ we find that
		\begin{equation}\label{equ2}
			f(x^{nm}\circ y^{mn})+ x^{nm}y^{mn}\in Z(\mathcal{A})\;\;\;\mbox{for all}\;\; x,y\in\mathcal{A}.
		\end{equation}
		Substituting $x$ for $y$ and $y$ for $x$ in $(\ref{equ2}),$ we obviously get
		\begin{equation}\label{equ3}
			f(y^{nm}\circ x^{mn})+ y^{nm}x^{mn}\in Z(\mathcal{A})\;\;\;\mbox{for all}\;\; x,y\in\mathcal{A}.
		\end{equation}
		and subtracting $(\ref{equ3})$ from $(\ref{equ2})$, we get
		\begin{equation}\label{equ4}
			[x^{nm},y^{nm}]\in Z(\mathcal{A})\;\;\mbox{for all}\;\; x,y\in\mathcal{A}.
		\end{equation}
		By \cite{Yood94} the Banach algebra $\mathcal{A}$ is commutative.
	\end{proof}
	
	\hspace{0.5cm} 
	
	\begin{app}
		Let $\mathcal{A}=\mathcal{\mathcal{M}}_{2}(\mathbb{C})$
		be the set of $2\times 2$ matrix with matrix addition and matrix multiplication.\\ 
		For $A =\begin{pmatrix}
			a & b\\
			c & d
		\end{pmatrix}\in \mathcal{A}$, define $\Vert A \Vert_2 =(\vert a \vert^2+ \vert b\vert^2+\vert c\vert^2 +\vert d \vert^2)^\frac{1}{2}$. Then ($\mathcal{A},\Vert .\Vert_2 )$ is a normed linear space. According to {\rm \cite{an} (Application 3.1 }\big),
		$\mathcal{H}=\left\{ \begin{pmatrix}
			e^{it} & 0\\
			0 & e^{-it}
		\end{pmatrix} \mid t\in\mathbb{R}\right \}$
		is an open subset of $\mathcal{A}$  and we have
		$
		[A^n,  B^m]=0$ for all  $A,B\in \mathcal{H}$ and for all $n,m\in \mathbb{N^*}.$ Hence, it follows that $\mathcal{A}$ is not a Banach algebra under the defined norm.
		
	\end{app}
	\hspace{0.5cm} 
	
	\noindent
	The following theorems are simple to establish using a proof similar to that one with a few minor modifications.
	\begin{theorem}\label{c}
		Let $\mathcal H_{1}$ and $\mathcal H_{2}$ be two non-void open subsets of  a unitary  prime Banach algebra $\mathcal{A}.$ If $\mathcal{A}$ admits a
		continuous endomorphism $f$ satisfying any one of the following conditions:
		\\
		1. For all $(x,y)\in \mathcal H_{1}\times \mathcal H_{2}$, there exists $(p,q)\in \mathbb N^{*}\times \mathbb N^{*}$ such that
		$f([x^{p},y^{q}])+ x^{p}y^{q}\in Z(\mathcal{A}), $\\
		2. For all $(x,y)\in \mathcal H_{1}\times \mathcal H_{2}$, there exists $(p,q)\in \mathbb N^{*}\times \mathbb N^{*}$ such that
		$f([x^{p},y^{q}])+ x^{p}\circ y^{q}\in Z(\mathcal{A}),$\\
		3. For all $(x,y)\in \mathcal H_{1}\times \mathcal H_{2}$, there exists $(p,q)\in \mathbb N^{*}\times \mathbb N^{*}$ such that
		$f([x^{p},y^{q}])- x^{p}y^{q}\in Z(\mathcal{A}),$\\
		4. For all $(x,y)\in \mathcal H_{1}\times \mathcal H_{2}$, there exists $(p,q)\in \mathbb N^{*}\times \mathbb N^{*}$ such that
		$f([x^{p},y^{q}])- x^{p}\circ y^{q}\in Z(\mathcal{A}),$\\
		then $\mathcal{A}$ must be commutative.
	\end{theorem}
	
	\hspace{0.5cm} 
	
	\begin{theorem}\label{d}
		Let $\mathcal{A}$ be a unitary  prime Banach algebra and $\mathcal{H}_{1},\mathcal{H}_{2}$  non-void open subsets of $\mathcal{A}.$
		If $\mathcal{A}$ admits a  continuous surjective endomorphism $ f$  satisfying any one of the following conditions:\\
		1. For all $(x,y)\in \mathcal H_{1}\times \mathcal H_{2}$, there exists $(p,q)\in \mathbb N^{*}\times \mathbb N^{*}$ such that $f(x^{p}y^{q})-[x^{p} , y^{q}]\in Z(\mathcal{A}),$\\
		2. For all $(x,y)\in \mathcal H_{1}\times \mathcal H_{2}$, there exists $(p,q)\in \mathbb N^{*}\times \mathbb N^{*}$ such that $f(x^{p}\circ y^{q})-x^{p}  y^{q}\in Z(\mathcal{A}),$\\
		then $ \mathcal{A} $ is commutative.
	\end{theorem}
	
	\hspace{0.5cm} 
	
	\begin{app}\label{a1}
		$\mathcal{A}=\mathcal{\mathcal{M}}_{2}(\mathbb{R})$
		endowed with usual matrix addition and  multiplication and the norm $\Vert .\Vert_{1}$ defined by $ \Vert A \Vert_{1}=\displaystyle\sum_{\substack{1\leq i,j \leq 2}}\vert a_{i,j}\vert $ for all $ A=(a_{i,j})_{1\leqslant i,j \leqslant2}\in\mathcal{A}$, is a unitary prime Banach algebra.
		Let $f$ be the application defined by 
		$f(M)=P^{-1}MP$
		where $P=\begin{pmatrix} 2 & 1 \\
			-1 & 1\end{pmatrix}$.\;We 
		have $f$ as a continuous automorphism.
		Let $\mathcal{H}$ be a nonempty open subset of $\mathcal{A}$ included in $Z(\mathcal{A}).$
		For all $ (A,B) \in \mathcal{H} \times \mathcal{H}  $ and for all $(p,q)\in\mathbb{N}^{*} \times \mathbb{N}^{*} $ we have
		$ A^{p}\in Z(\mathcal{A})$ and $ B^{q}\in Z(\mathcal{A}),$ then  $f(A^{p}B^{q})- [A^{p}, B^{q}]\in Z(\mathcal{A}).$
		By Theorem $\ref{d},$ we conclude that $\mathcal{A}$ is commutative. Contradiction, consequently $\mathcal{H}=\emptyset$.\\
		We conclude that the only open subset included in $Z(\mathcal{A})$ is the empty set.
	\end{app}
	
	\hspace{0.5cm} 
	
	
	\hspace{0.5cm} 
	The Jordan product or the multiplicative law is represented by the letter  $T$  in the following theorem.
	\begin{theorem}\label{m}
		Let $\mathcal{A}$ be a unitary  prime Banach algebra and $\mathcal H_{1},\mathcal H_{2}$ non-void open subsets of $\mathcal{A}.$ Let $f$ be a continuous surjective endomorphism such that:
		For all $(x,y)\in \mathcal H_{1}\times \mathcal H_{2}$, there exists $(p,q)\in \mathbb N^{*}\times \mathbb N^{*}$ and $f(x^{p}Ty^{q})\in Z(\mathcal{A}).
		$
		Then $\mathcal{A}$ is commutative.
	\end{theorem}
	\begin{proof}
		For $(p,q)\in\mathbb N^{*}\times \mathbb N^{*},$ setting
		$$
		O_{p,q}=\{(x,y)\in \mathcal A^{2}\; | \;  f(x^{p}Ty^{q})\notin Z(\mathcal{A} )\}\;\;\mbox{and}\;\;
		F_{p,q}=\{(x,y)\in \mathcal A^{2}\; |\;  f(x^{p}Ty^{q})\in Z(\mathcal{A} )\}.
		$$
		We claim that  $O_{p,q}$ is open in $\mathcal{A} \times \mathcal{A}$, that is its complement $F_{p,q}$ is closed. For this, consider a sequence $((x_{k},y_{k}))_{k\in\mathbb{N}}\subset F_{p,q}$ converging to $(x,y)\in\mathcal A^{2},$
		since $((x_{k},y_{k}))_{k\in\mathbb{N}}\subset F_{p,q}$  so that
		\begin{center}
			$f((x_{k})^{p}T(y_{k})^{q})\in Z(\mathcal{A})$ for all $k\in\mathbb{N}.$
		\end{center}
		$f$ being continuous, then the sequence  $(f((x_{k})^{p}T(y_{k})^{q}))_{k\in\mathbb{N}}$ converging to
		$ f(x^{p}Ty^{q}) $, which because of $Z(\mathcal{A})$ is closed, assures that $ f(x^{p}Ty^{q})\in Z(\mathcal{A}).$
		Accordingly, $ (x,y)\in F_{p,q} $ and thus $ F_{p,q}$ is closed which means that  $O_{p,q}$ is open.\\
		If every $ O_{p,q} $ is dense, then their intersection is also dense by
		Baire's category theorem, which contradict with $(\cap O_{p,q})\cap (\mathcal H_{1}\times \mathcal H_{2})= \varnothing$.
		Hence there are two positive integers $n$ and $m$ such that $ O_{n,m} $ is not a dense set and there exists a non-void open subset $ O \times O^{'}$ in $ F_{n,m}$  such that :
		\begin{center}
			$f(x^{n}Ty^{m})\in Z (\mathcal{A})$ for all $(x,y)\in O\times O^{'}.$
		\end{center}
		Let $y_{0} \in O $ and $y\in\mathcal{A},$ since $y_{0}+t y\in O$ for all sufficiently small $ t\in\mathbb{R} $ then \\ $f(x^{n}T(y_{0}+t y)^{m})\in Z(\mathcal{A})$.
		The expression $f(x^{n}T(y_{0}+t y)^{m}) $ can be written as:
		$$f(x^{n}T(y_{0}+ty)^{m})=A_{m,0}(x,y_0,y)+A_{m-1,1}(x,y_0,y)t + A_{m-2,2}(x,y_0,y)t^{2}+...+A_{0,m}(x,y_0,y)t^{m}.$$
		Consequently, given $x\in O$ by Lemma \ref{l1}  we have $f(x^{n}Ty^{m})\in Z(\mathcal{A})$ for all $y\in \mathcal{A}.$
		Reversing the roles  $O$ and $O^{'},$ we easily get
		$$
		f(x^{n}Ty^{m})\in Z(\mathcal{A}) \;\;\mbox{for all} \;\;(x,y )\in\mathcal A^2.
		$$
		If $T$ denotes the Jordan product (resp. the ordinary product), then the last expression becomes  $f(x^{n+m})\in Z(\mathcal{A})$ (resp. $2f(x^{n+m})\in Z(\mathcal{A}))\;$  for all $x\in\mathcal{A}.$ Hence, in both cases we obtain
		$f(x^{n+m})\in Z(\mathcal{A}) $ for all $x\in\mathcal{A}.$ Arguing in a similar way as in the proof of  Theorem $\ref{a},$
		$\mathcal{A}$ is commutative.
	\end{proof}
	\noindent
	\hspace{0.5cm} 
	
	With the following corollary, we bring this discussion to an end.
	\begin{corollary}\label{d1}
		Let $\mathcal{A}$ be a unitary  prime Banach algebra and $D$  a part dense in  $\mathcal{A}.$
		If $\mathcal{A}$ admits a continuous surjective endomorphism $ f $ such that: There exists $(p,q)\in \mathbb N^{*}\times \mathbb N^{*}$ and $f(x^{p}y^{q})+[x^{p},y^{q}]\in Z(\mathcal{A})$ for all $x,y\in D,$
		then $\mathcal{A}$ is commutative.
	\end{corollary}
	\begin{proof}
		Let $ x,y\in\mathcal{A},$ there exist  two  sequences $(x_{k})_{k\in\mathbb{N}}$ and $(y_{k})_{k\in\mathbb{N}}$  in $ D $ converging to $ x$ and $y.$
		Since $(x_{k})_{k\in\mathbb{N}}\subset D $ and $(y_{k})_{k\in\mathbb{N}}\subset D $, then
		\begin{center}
			$f((x_{k})^{p}(y_{k})^{q})+[(x_{k})^{p},(y_{k})^{q}]\in Z(\mathcal{A})$ for all $k\in\mathbb{N}.$
		\end{center}
		Given that $d$ is continuous, the sequence  $(f((x_{k})^{p}(y_{k})^{q})+[(x_{k})^{p},(y_{k})^{q}])_{k\in\mathbb{N}}$ converges to
		$ f(x^{p}y^{q})+[x^{p},y^{q}] $, knowing that $Z(\mathcal{A})$ is closed, then $ f(x^{p}y^{q})+[x^{p},y^{q}]\in Z(\mathcal{A}).$
		We conclude that:
		there exists $(p,q)\in\mathbb N^{*}\times \mathbb N^{*}$ and  $f(x^{p}y^{q})+[x^{p},y^{q}]\in Z(\mathcal{A})$ for all $x,y\in\mathcal{A}.$
		By Theorem $\ref{a},$ we achieve the desired result.
	\end{proof}
	{\bf 3. Some examples}

	\hspace{0.5cm} 
	The following example shows that the surjectivity of the endomorphism $f$ is necessary for the hypotheses of Theorem \ref{m}.
	\begin{example}
		Consider the ring  $\mathcal{ A}=
		\left\{
		\begin{pmatrix}
			a & 0\\
			a&b
		\end{pmatrix} \;|\; a,b\in \mathbb{R} \right\}$, provided with matrix addition, matrix multiplication and the norm defined by $ \left\Vert \begin{pmatrix}
			a&0 \\
			a&b\\
		\end{pmatrix} \right\Vert_{1}= |a|+|a| +|b|$. Then, clearly $\mathcal{ A}$ is prime Banach algebra.\;Let $f$ be the
		additive mapping is defined by
		$$f(\begin{pmatrix}
			a&0 \\
			a&b\\
		\end{pmatrix})=\begin{pmatrix}
			a&0 \\
			0&a\\
		\end{pmatrix}$$
		for all $(a,b)\in \mathbb R^2$. It is obvious to see that $f$ is a continuous  endomorphism non-surjective on $\mathcal{ A}$ and we have $f(A^n B^m)\in Z(\mathcal{ A})$  ( $f(A^n \circ B^m)\in Z(\mathcal{ A})$) for every $A,B\in \mathcal{ A}$ and for every $n,m\in\mathbb N^*$. However, $\mathcal A$ is not commutative.
	\end{example}
	\hspace{0.5cm} 
	\noindent
	The following example shows that the primeness of $\mathcal{A}$ in our theorems can not be omitted.
	\begin{example}
		Let  
		$\mathcal{ A}=
		\left\{
		\begin{pmatrix}
			0 & a&b \\
			0 & 0&c\\
			0&0&0
		\end{pmatrix} \;|\; a,b,c \in \mathbb{R} \right\}.$
		Then $\mathcal{ A}$ is a Banach algebra under the norm defined by
		$ \left\Vert \begin{pmatrix}
			0 & a&b \\
			0 & 0&c\\
			0&0&0
		\end{pmatrix} \right\Vert_{1}= |a|+|b| +|c|$. We also have, if $a,b,c\in\mathbb{R}$, then
		$$
		\begin{pmatrix}
			0 & 0&1 \\
			0 & 0&0\\
			0&0&0
		\end{pmatrix} \begin{pmatrix}
			0 & a&b \\
			0 & 0&c\\
			0&0&0
		\end{pmatrix} \begin{pmatrix}
			0 & 0&1 \\
			0 & 0&0\\
			0&0&0
		\end{pmatrix}=\begin{pmatrix}
			0 & 0&0 \\
			0 & 0&0\\
			0&0&0
		\end{pmatrix}
		$$
		which proves that $\mathcal{A}$ is a not prime Banach Algebra.
		Furthermore, for all $(A,B) \in\mathcal A^{2}$ we have :\\
		1. $Id(A^{3}B^{3})+[A^{3},B^{3}]\in Z(\mathcal{A})$ and $Id(A^{3}B^{3})-[A^{3},B^{3}]\in Z(\mathcal{A}),$  \\
		2. $Id(A^{3}\circ B^{3})+A^{3}B^{3}\in Z(\mathcal{A})$ and $Id(A^{3}\circ B^{3})-A^{3}B^{3}\in Z(\mathcal{A}),$  \\
		3. $Id([A^{3},B^{3}])+A^{3}B^{3}\in Z(\mathcal{A})$ and $Id([A^{3},B^{3}])-A^{3}B^{3}\in Z(\mathcal{A}),$ \\
		4. $Id([A^{3},B^{3}])+A^{3}\circ B^{3}\in Z(\mathcal{A})$ and $Id([A^{3},B^{3}])-A^{3}\circ B^{3}\in Z(\mathcal{A}),$  \\
		5. $Id(A^{3}B^{3})\in Z(\mathcal{A}),$ \\
		6. $Id(A^{3}\circ B^{3})\in Z(\mathcal{A}).$\\
		As a result, $\mathcal{A}$ is not commutative even if $Id$ meets the hypothesis of all Theorems.
	\end{example}
	
	\hspace{0.5cm}
	
	\noindent
	\noindent
	\noindent
	The following example indicates that our Theorems' premise that "$\mathcal{H}_{1}$ and $\mathcal{H}_{2}$ are open" cannot be modified.
	\begin{example}
		Consider the real  prime  Banach algebra $\mathcal{A}=\mathcal{\mathcal{M}}_{2}(\mathbb{R})$ endowed with $\Vert .\Vert_{1}$ defined  in Example 1.
		Let
		$ \mathcal{F}=\left\{
		\begin{pmatrix}
			0 & a \\
			0 & 0
		\end{pmatrix} \;|\; a \in \mathbb{R}  \right\}.$ We claim that
		$\mathcal{F}$
		is a closed subspace  of $\mathcal{A}$. Indeed, we consider a sequence
		$(A_n)_{n \in \mathbb{N^{*}}}$ in $\mathcal{F}$
		converges to $A\in\mathcal{A}$.\\
		Since $(A_n)_{n \in \mathbb{N^{*}}}\subset \mathcal{F}$, therefore, there is $(a_n)_{n \in \mathbb{N^{*}}}\subset \mathbb{R}$ such that 
		$$ A_n= \begin{pmatrix} 0 & a_n \\
			0 & 0\end{pmatrix} \;\; \forall n \in \mathbb{N^*} .$$
		As the sequence $(A_n)_{n \in \mathbb N^{*}}$ is convergent, then the sequence  $(a_n)_{n \in \mathbb{N^{*}}}$ is also convergent to a real $a$ and $A=\begin{pmatrix} 0 & a \\
			0 & 0\end{pmatrix}.$
		Let $f$ be the  automorphism of $\mathcal{A}$ defined by  $f(M)=QMQ^{-1}$ for all $M\in\mathcal{A}$ where  
		$Q=
		\begin{pmatrix}
			2 & 1 \\
			1 & 2
		\end{pmatrix}.$  It is easy to see that $f$ is a  continuous automorphism. Moreover,
		for $ A \in\mathcal{F}$  and $B \in
		\mathcal{F} $, we  have
		$$
		A^{2}B^{2}=
		\begin{pmatrix}
			0 & 0 \\
			0 & 0
		\end{pmatrix}
		,\;\;
		A^{2}\circ B^{2}=
		\begin{pmatrix}
			0 & 0 \\
			0 & 0
		\end{pmatrix},\;\;
		[A^{2},B^{2}]=
		\begin{pmatrix}
			0 & 0 \\
			0 & 0
		\end{pmatrix}$$
		consequence of which
		$$
		f(A^{n}\circ B^{m})=f([A^{n},B^{m}])=f(A^{n}B^{m})=\begin{pmatrix} 0 & 0 \\
			0 & 0\end{pmatrix}.
		$$
		Using the fact that \\
		1. $f(A^{n}B^{m})+[A^{n},B^{m}]\in Z(\mathcal{A})$ and $f(A^{n}B^{m})-[A^{n},B^{m}]\in Z(\mathcal{A}),$ \\
		2. $f(A^{n}\circ B^{m})+A^{n}B^{m}\in Z(\mathcal{A})$ and $f(A^{n}\circ B^{m})-A^{n}B^{m}\in Z(\mathcal{A}),$  \\
		3. $f([A^{n},B^{m}])+A^{n}B^{m}\in Z(\mathcal{A})$ and $f([A^{n},B^{m}])-A^{n}B^{m}\in Z(\mathcal{A}),$   \\
		4. $f([A^{n},B^{m}])+A^{n}\circ B^{m}\in Z(\mathcal{A})$ and $f([A^{n},B^{m}])-A^{n}\circ B^{m}\in Z(\mathcal{A}),$   \\
		5. $f(A^{n}B^{m})\in Z(\mathcal{A}),$  \\
		6. $f(A^{n}\circ B^{m})\in Z(\mathcal{A}).$ 
		\hspace{0.5cm} 
		
		It follows that $f$ satisfies the hypotheses of our theorems, however $\mathcal{A}$ is not commutative.
	\end{example}
	The example below demonstrates that the conclusion of our theorems does not necessarily have to be true for a Banach algebra over $\mathbb{F}_3$.
	\begin{example}
		Let $\mathbb{F}_3$ be the field $\mathbb{Z}/3\mathbb{Z}$ and $\mathcal{A}=\mathcal{M}_{2}(\mathbb{F}_3)$. Clearly, $\mathcal{A}$ is  2-torsion free prime  Banach algebra aver $\mathbb{F}_3$  with norm $\Vert . \Vert_{1}$  defined by 
		$ \Vert A \Vert_{1}= 
		\displaystyle \sum_{1\leq i,j \leq 2} \vert a_{i,j}\vert $ for  $ A=(a_{i,j})_{1\leqslant i,j \leqslant2}\in\mathcal{A}$  where
		$\vert. \vert$  is the norm defined on $\mathbb{F}_3$ by:
		\begin{center}
			$\vert \overline{0}\vert =0 $, $\vert \overline{1}\vert =1 $ and $\vert \overline{2}\vert =2 .$
		\end{center}
		Observe that
		$\mathcal{ H}=
		\left\{
		\begin{pmatrix}
			a & 0 \\
			0 & a
		\end{pmatrix} \;|\; a \in \mathbb{Z}/3\mathbb{Z} \right\}$ is an open subset of $\mathcal{A}.$ Indeed, for $A \in \mathcal{ H}$  the open ball $B (A,1)=\{X\in \mathcal{A}$ such that $ \Vert A-X \Vert_{1}< 1 \}=\{A\} \subset \mathcal {H}.$ \\
		Furthermore, for  all  $(n,m)\in\mathbb N^{*^2}$ and for all $(A,B) \in\mathcal H^{2}$ we have :\\
		1. $Id(A^{n}B^{m})+[A^{n},B^{m}]\in Z(\mathcal{A})$ and $Id(A^{n}B^{m})-[A^{n},B^{m}]\in Z(\mathcal{A}),$  \\
		2. $Id(A^{n}\circ B^{m})+A^{n}B^{m}\in Z(\mathcal{A})$ and $Id(A^{n}\circ B^{m})-A^{n}B^{m}\in Z(\mathcal{A}),$  \\
		3. $Id([A^{n},B^{m}])+A^{n}B^{m}\in Z(\mathcal{A})$ and $Id([A^{n},B^{m}])-A^{n}B^{m}\in Z(\mathcal{A}),$ \\
		4. $Id([A^{n},B^{m}])+A^{n}\circ B^{m}\in Z(\mathcal{A})$ and $Id([A^{n},B^{m}])-A^{n}\circ B^{m}\in Z(\mathcal{A}),$  \\
		5. $Id(A^{n}B^{m})\in Z(\mathcal{A}),$ \\
		6. $Id(A^{n}\circ B^{m})\in Z(\mathcal{A}).$\\
		This means that $Id$ satisfies the hypotheses of our theorems but $\mathcal{A}$ is not commutative.
	\end{example}

	\noindent
	\begin{example}
		Let  $\mathcal{A}=\mathcal{\mathcal{M}}_{2}(\mathbb{C})$. 
		$\mathcal{A}$ is a prime complex Banach algebra non-commutative. 
		Furthermore we have $\mathcal H=\{\begin{pmatrix} a& b \\
			c & d \end{pmatrix}\in\mathcal{A}$ such that $ad-bc\neq 0 \}$ is open in $\mathcal{A}$.\\
		Take  $A=\begin{pmatrix} 1& 1 \\
			0 & 1 \end{pmatrix}\in \mathcal H $ and $B=\begin{pmatrix} 1& 0 \\
			0 & 1 \end{pmatrix}\in H,$
		then $A^{p}=\begin{pmatrix} 1& p \\
			0 & 1 \end{pmatrix}$ and $B^{q}=\begin{pmatrix} 1& 0 \\
			0 & 1 \end{pmatrix}$
		therefore $Id(A^{p} B^{q})+[A^{p}, B^{q}] =\begin{pmatrix} 0& 2p \\
			0 & 0 \end{pmatrix}\notin\ Z(\mathcal{A})$ for all $p,q \in\mathbb{N^{*}}.$\\
		This show that the condition $f(A^{p}\circ B^{q})+[A^{p}, B^{q}]\in Z(\mathcal{A})$ "for all $A\in \mathcal H$" 
		in Theorem $\ref{a}$, is not superfluous.
	\end{example}

\end{document}